\documentclass[a4paper]{imanum}
\jno{drnxxx}
\received{}

\usepackage{amsmath,amsfonts,amssymb,amsthm,graphicx,mathtools}

\numberwithin{equation}{section}
\DeclareMathOperator{\Dim}{dim}
\DeclareMathOperator{\Span}{span}
\DeclareMathOperator{\Diag}{diag}

\newcommand{\bo}[1]{\mathbf{#1}} 
\newcommand{\mrm}[1]{\mathrm{#1}}
\newcommand{\dd}{\mrm{d}}
\newcommand{\abs}[1]{\lvert#1\rvert}
\newcommand{\tends}{\rightarrow}
\newcommand{\norm}[1]{\lVert#1\rVert}

\newcommand{\R}{\mathbb{R}}

\newcommand{\calD}{\mathcal{D}}
\newcommand{\calF}{\mathcal{F}}

\newcommand{\calP}{\mathcal{P}}

\newcommand{\calI}{\mathcal{I}}

\newcommand{\calL}{\mathcal{L}}
\newcommand{\calB}{\mathcal{B}}

\newcommand{\calG}{\mathcal{G}}

\newcommand{\calV}{\mathcal{V}}
\newcommand{\calX}{\mathcal{X}}
\newcommand{\Om}{\Omega}

\newcommand{\bp}[1]{\mathsf{#1}}
\newcommand{\bv}{\mathsf{V}}
\newcommand{\cB}{\mathcal{B}}
\newcommand{\cF}{\mathcal{F}}

\newcommand{\Id}{\mathrm{Id}}
\newcommand{\Vp}{\calV_p}

\begin{document}

\title{Robust and efficient preconditioners for the discontinuous Galerkin time-stepping method}

\shorttitle{Robust and efficient preconditioners for the DG time-stepping method}

\author{{\sc Iain~Smears\thanks{Corresponding author. Email:iain.smears@inria.fr}\\[2pt]
INRIA Paris, 2 Rue Simone Iff, 75012 Paris, France}}
\shortauthorlist{I.~Smears}

\maketitle

\begin{abstract}
{The discontinuous Galerkin time-stepping method has many advantageous properties for solving parabolic equations. However, it requires the solution of a large nonsymmetric system at each time-step.
This work develops a fully robust and efficient preconditioning strategy for solving these systems. Drawing on parabolic inf-sup theory, we first construct a left preconditioner that transforms the linear system to a symmetric positive definite problem to be solved by the preconditioned conjugate gradient algorithm. We then prove that the transformed system can be further preconditioned by an ideal block diagonal preconditioner, leading to a condition number bounded by $4$ for any time-step size, any approximation order and any positive-definite self-adjoint spatial operators.
Numerical experiments demonstrate the low condition numbers and fast convergence of the algorithm for both ideal and approximate  preconditioners, and show the feasibility of the high-order solution of large problems.}
{discontinuous Galerkin; time discretizations; parabolic PDE; preconditioning; conjugate gradient algorithm.}
\end{abstract}

\section{Introduction}\label{sec:introduction}
The discontinuous Galerkin (DG) time-stepping method is a single-step implicit scheme defined by a variational temporal discretization of parabolic evolution equations that generalizes the backward Euler method to higher-order approximations \cite[]{Delfour1981,ErikssonJohnsonThomee1985,Jamet1978,Hulme1972}. For an introduction to this time discretization scheme in the context of abstract parabolic problems, we refer the reader to \cite[]{Thomee2006}. In certain cases, it coincides with the Radau IIA Implicit Runge--Kutta (IRK) schemes and the subdiagonal Pad\'e approximations to the exponential function \cite[]{Axelsson1969,Hairer2010,Makridakis2006}. It can be coupled with standard spatial discretization schemes, such as finite difference, finite element or spectral methods; in particular, when coupled to a spatial finite element method (FEM), it leads to a tensor-product space-time FEM.

The DG time-stepping method features many advantages that make it an attractive choice for solving parabolic problems.
First, it permits arbitrarily high-order approximation to the solution, along with superconvergence at the time-step nodes \cite[]{Akrivis2004,Chrysafinos2006,Makridakis1997}, see also \cite[]{Schotzau2000} for optimal-order a priori error estimates in natural norms with explicit dependence on the polynomial degree. Unlike linear multistep schemes, it is thus not constrained by the Dahlquist barrier theorem and it does not require an auxiliary scheme to compute the first few solution values.
Furthermore, the DG time-stepping method allows fully variable time-step sizes, approximation orders, and even spatial mesh refinement/coarsening between time-steps; it is thus well-suited for adaptive algorithms driven by rigorous a posteriori error estimates \cite[]{Akrivis2009,Eriksson1991,Eriksson1995,Makridakis2006,Schotzau2010}.
The DG time-stepping method also permits the temporal version of $hp$-refinement, where one varies both the time-step size as well as the approximation order on each time-step, thereby yielding exponential convergence rates for a broad class of solutions with singularities induced by the initial datum or by the source term~\cite[]{Schotzau2000}. In these applications, it is common to use high temporal polynomial degrees in order to match the high-order spatial approximation, see for instance the experiments in~\cite[]{Schotzau2000,vonPetersdorff2004,Schotzau2010,WerderGerdesSchotzauSchwab2001}.

In applying the DG time-stepping method in practice, we are faced with the challenge of solving a large nonsymmetric linear system at each time-step.
In this work, we are interested in developing preconditioned iterative methods for solving these linear systems. We focus here on the DG time-stepping method in the context of an abstract semi-discrete evolution problem of the form
\begin{equation}\label{eq:intro_evolution}
\begin{aligned}
M\, U^\prime(t) + A\, U(t)  &= f(t) 	& &\text{in } (0,T), \\
\end{aligned}
\end{equation}
where the solution $U\colon [0,T] \tends \bv$, with $\bv$ a finite dimensional space, where $M$ and $A$ are symmetric positive definite matrices.
Semi-discrete problems of the form~\eqref{eq:intro_evolution} arise from the application of a wide range of spatial discretizations, including conforming and nonconfirming finite elements, finite differences, and spectral methods, to a broad class of parabolic problems, such as general self-adjoint second-order and fourth-order parabolic partial differential equations (PDE).
We emphasize that our results are valid for general symmetric positive definite matrices $M$ and $A$, although the reader may consider the heat equation as a concrete example, with $M$ and $A$ respectively representing the mass and stiffness matrices for a suitable approximation space $\bv$.
The DG time-stepping method applied to the evolution problem~\eqref{eq:intro_evolution} leads to a sequence of linear systems of the general block form
\begin{equation}\label{eq:dg_linear_system}
\begin{bmatrix}
b_{00}\, M + c_{00}\, \tau \,A &  \cdots 	& b_{0p}\, M + c_{0p}\,\tau\,  A \\
\vdots & \ddots & \vdots \\
b_{p0}\, M + c_{p0}\,\tau\, A & \cdots & b_{pp} \,M + c_{pp}\,\tau \,A
\end{bmatrix}
\begin{bmatrix}
\bp u_0 \\ \vdots \\ \bp u_p	
\end{bmatrix}
= \begin{bmatrix} \bp f_0 \\ \vdots \\ \bp f_p	\end{bmatrix},
\end{equation}
where $\tau$ denotes the time-step size, the polynomial degree $p$ defines the approximation order of the scheme, the solution coefficients $\bp u_k\in \bv$ for $k=0,\dots,p$, and, after mapping the time-step interval to the reference interval $(-1,1)$, we have
\begin{equation}
\begin{aligned}
	b_{jk} \coloneqq \int_{-1}^1 \phi_k^\prime\, \phi_j\,\dd s + \phi_k(-1)\phi_j(-1),
& & & c_{jk} \coloneqq \frac{1}{2} \int_{-1}^1 \phi_k\,\phi_j\,\dd s,
\end{aligned}	
\end{equation}
where $\{\phi_k\}_{k=0}^p$ is a chosen basis of $\calP_p$ the space of real-valued polynomials of degree at most $p$. For the case $p=0$, the DG method reduces to the backward Euler method and the system is therefore symmetric. However, for $p\geq 1$, the system matrix $\bo B$ of \eqref{eq:dg_linear_system} is nonsymmetric with dimension~$\Dim \bv \times (p+1)$, which is considerably larger than for linear multistep methods; even for moderate sizes of $A$ and $M$, standard direct solution algorithms can be prohibitively expensive. Unlike the block triangular linear systems obtained from Diagonally IRK (DIRK) and  Singly Diagonally IRK (SDIRK) schemes (see \cite[]{Hairer2010} and the references therein), the system~\eqref{eq:dg_linear_system} does not immediately reveal any simple structure offering a straightforward solution algorithm. See also \cite[p.~129--130]{Hairer2010} for a discussion of SIRK methods.

Since the DG time-stepping method is connected to the Radau IIA IRK scheme, it is interesting to relate our approach to the literature on solving the systems of IRK schemes.
In particular, one of the earliest approaches \cite[]{Butcher1976,Bickart1977} for solving the linear systems of general IRK schemes, such as ~\eqref{eq:dg_linear_system}, is based on transforming the system matrix to a block-diagonal matrix with blocks of the form $M + \frac{\tau}{\mu_i} A $, where the $\{\mu_i\}_{i=0}^p$ denote here the generalized eigenvalues of $(b_{jk})$ with respect to $(c_{jk})$, and where the transformation is given by the corresponding eigenvectors. This leads to $p+1$ independent smaller linear systems that can be solved in parallel. It turns out that for the DG time-stepping method, the generalized eigenvalues $\mu_i$ are complex numbers related to the roots of the denominator of a rational Pad\'e approximation to the exponential function. Therefore, the resulting transformed system is complex-valued and non-Hermitian despite $M$ and $A$ being symmetric. In addition to the increased computational cost of complex arithmetic, this approach has an important shortcoming in terms of robustness and numerical stability for high-order approximations, due to ill-conditioning of the eigenvector transformations, as pointed out in \cite[Remark~5.4]{vonPetersdorff2004}.
This issue can be avoided by employing alternative factorizations, at the expense of the block-diagonal structure of the transformed problem: \cite{Schotzau2000} propose a factorization based on the Schur decomposition theorem, leading to a block-triangular complex transformed problem; the solution is then obtained by solving $p+1$ complex non-Hermitian systems in sequence.

Preconditioned iterative methods offer an alternative approach to decoupling the system by complex transformations. In this direction, \cite{Richter2013} propose a linear iterative fixed point scheme based on an approximation of the block LU factorization of $\bo B$, and they analyse the contraction rates of their method for $p\leq 3$.
An alternative approach is to apply directly a preconditioned Krylov subspace method to \eqref{eq:dg_linear_system}, as proposed in 
\cite{Mardal2007} propose a block-diagonal preconditioner for IRK schemes to be used with GMRES. They show that the preconditioner is robust with respect to the time-step size $\tau$, but their experiments show that it is not robust with respect to $p$.
\cite{Weller2015} develop a preconditioner specifically for the DG time-stepping method for $p=1$ (as well as for the related continuous Galerkin method of same size) based on an approximate Schur complement preconditioner for one of the unknown coefficients $\bp u_j$ in \eqref{eq:dg_linear_system}.
It is thus apparent from these references that finding preconditioning strategies for the DG time-stepping method that are robust with respect to the polynomial degree $p$ has been a challenging open problem.

In this work, we propose a robust and efficient preconditioned iterative method for the DG time-stepping method. 
Instead of focusing exclusively on the block structure of \eqref{eq:dg_linear_system}, our approach draws upon the inf-sup analysis of the method and the underlying continuous analysis of parabolic PDEs, and we take advantage of the variational structure of the DG time-stepping method in an essential way.
First, in section~\ref{sec:left_precond}, we construct and apply a left-preconditioner $\bo P^\top$ to the linear system $\bo B \bo u = \bo f$ given by \eqref{eq:dg_linear_system}, resulting in a preconditioned system $\bo L \bo u = \bo g$ with the key benefit that the transformed matrix~$\bo L\coloneqq \bo P^\top \bo B$ is symmetric positive definite.
We immediately point out that $\bo P^\top \neq \bo B^\top$, i.e.\ we are not forming the normal equations of the system.
Instead, our construction of $\bo P^\top$ is motivated by parabolic inf-sup theory, and we show that the matrix $\bo L$ represents the discrete version of the natural parabolic energy norm of the underlying evolution equation: for example, in the context of second-order parabolic PDEs, the matrix~$\bo L$ is a discrete Gram matrix for the natural solution space $H^1(H^{-1})\cap L^2(H^1)$; we refer the reader to \cite[]{Wloka1987} for an introduction to the continuous analysis of parabolic problems.
The transformed symmetric positive definite system can therefore be solved by the preconditioned conjugate gradient~(PCG) algorithm \cite[]{Hestenes1952,MalekStrakos2015,Wathen2015}, which, in our case, minimizes the error in the physically relevant norm over a Krylov subspace.
In order to obtain the fast and robust convergence of the PCG algorithm, in section~\ref{sec:spect_equiv},  we construct  a spectrally equivalent preconditioner $\bo H$ for $\bo L$, such that the condition number of the preconditioned system satisfies
\begin{equation}\label{eq:condition_number_bound}
 	\kappa(\bo H^{-1} \bo L) \leq 4,
\end{equation}
independently of all parameters $\tau$, $p$, $M$ and $A$.
Therefore, the preconditioner is fully \emph{robust} with respect to all problem and discretization parameters, including the polynomial degree. Furthermore, the preconditioners are \emph{efficient}, firstly in the sense of guaranteeing the fast convergence of the PCG algorithm in the physically relevant norm, and secondly in the sense of computational cost, for the following reasons.
In section~\ref{sec:implementation} we show that the preconditioners are well-suited for parallelization over the blocks and involve only simpler matrices of the form $A$ and $M+\mu A$ with real positive $\mu$ for which efficient solvers are often available.
Furthermore, we show experimentally in section~\ref{sec:num_exp} that the ideal preconditioners can be approximated in practice by cheap spectrally equivalent approximations, such as a small number of multigrid $V$-cycles.
We refer the reader to \cite[p.~367]{Wathen2015} and \cite[p.~705]{Hiptmair2006} for further discussions of the notion of efficiency of preconditioners.

This paper is organized as follows: after introducing in detail the DG time-stepping method in section~\ref{sec:preliminaries}, we present the preconditioning strategy in section~\ref{sec:preconditioners}, where we construct the preconditioners $\bo P^\top$ and $\bo H$, and where we establish the condition number bound~\eqref{eq:condition_number_bound}. Section~\ref{sec:implementation} considers the efficient implementation of the method, and section~\ref{sec:num_exp} presents the results of numerical experiments testing the robustness and efficiency of the preconditioners.

\section{Preliminaries}\label{sec:preliminaries}

In this section, we introduce in detail the DG time-stepping method in the context of self-adjoint semi-discrete dissipative evolution equations.
We also introduce two key ingredients in our approach. The first ingredient is the well-known temporal reconstruction operator commonly used in a posteriori analysis \cite[]{Makridakis2006}, which is the subject of section~\ref{sec:reconstruction}. The second ingredient is a spectral equivalence result for preconditioners from \cite[]{Pearson2012}, which we present in section~\ref{sec:negative_norm}.

\subsection{Approximation space}\label{sec:approximation_space} Let $\bv$ denote a finite dimensional real vector space, equipped with a given basis $\{\bp e_i\}_{i=1}^{\Dim \bv}$.
Let $\bv$ be equipped with two inner products $(\cdot,\cdot)_{M}$ and $(\cdot,\cdot)_{A}$, and let $M$ and $A$ be their matrix representations, given by $M_{ij} \coloneqq (\bp e_i,\bp e_j)_M$ and  $A_{ij}\coloneqq (\bp e_i,\bp e_j)_{A}$ for all $i$, $j=1,\dots,\Dim \bv$. The inner products $(\cdot,\cdot)_M$ and $(\cdot,\cdot)_A$ induce the norms $\norm{\cdot}_M$ and $\norm{\cdot}_A$ on $\bv$.
Let $\calP_p$ denote the space of real-valued polynomials of degree at most $p$, and let $\Vp$ be the space of $\bv$-valued polynomials of a single real variable with degree at most $p$. 
For example, if $\{\phi_j\}_{j=0}^p$ is a basis of $\calP_p$, then every $v\in \Vp$ is of the form 
\[
v\colon s\mapsto v(s) = \sum_{j=0}^{p} \bp{v}_j\, \phi_j(s),
\]
with coefficients $\bp v_j \in \bv$ for each $j=0,\dots,p$. 
We gather these coefficients in the vector $\bo v =(\bp v_0,\dots,\bp v_p) \in \bv^{p+1}$.
It follows that  $\Dim \Vp $, the dimension of the space $\Vp$, is equal to $ (p+1)\times \Dim \bv$.
An equivalent point of view is to consider $\Vp$ as the tensor-product space derived from $\calP_p$ and $\bv$.

\begin{remark}\label{rem:basis}
In this setting, it is natural to view functions in $\Vp$ as mappings from time into $\bv$, and as a slight abuse of standard terminology, we will say that $\{\phi_j\}_{j=0}^p$ forms a basis of $\Vp$. Of course, this must be interpreted as the lengthier statement that $\{\bp e_i \phi_j\}_{i=1,\dots,\Dim \bv}^{j=0,\dots,p}$ forms a basis of $\Vp$, where $\{\bp e_i\}_{i=1}^{\Dim\bv}$ is a basis of~$\bv$.
\end{remark}

If $L$ is a linear operator between $\bv$ and either itself or its dual $\bv^*$, then we can extend $L$ to $\Vp$ by applying it coefficient-wise: 
\begin{equation}
\begin{aligned}
	L v(s) \coloneqq \sum_{j=0}^p \left(L \bp v_j\right) \phi_j(s) & & &\forall\, v=\sum_{j=0}^p \bp v_j \phi_j \in \Vp.
\end{aligned}
\end{equation}
The inner products $(\cdot,\cdot)_M$ and $(\cdot,\cdot)_A$ also extend to $\Vp$ in the natural way.
Likewise, if $ \pi \colon \calP_p \tends \calP_p$ is a linear operator then we define $\pi v \coloneqq \sum_{j=0}^p \bp v_j  \left(\pi\phi_j\right)$. For instance, we define the time derivative $v^\prime$ of $v\in \Vp$ by $v^\prime \coloneqq \sum_{j=0}^p \bp v_j \,\phi_j^\prime $.

Let $(\cdot,\cdot)_{L^2}$ and $\norm{\cdot}_{L^2}$ denote respectively the $L^2$-inner product and $L^2$-norm over the interval $(-1,1)$.
Let $\{L_k\}_{k\geq 0}$ denote the set of Legendre polynomials, as defined for instance in \cite[Sec.~8.9]{TOISP}. The Legendre polynomials are orthogonal in the $L^2$-inner product: for all $k$, $j\geq 0$, 
\begin{equation}
\left(L_k,L_j\right)_{L^2} = \frac{2}{2 k+1} \delta_{kj}.
\end{equation}
Furthermore, $L_k(1) = 1$ and $L_k(-1)=(-1)^k$ for all $k\geq 0$.

\subsection{The DG time-stepping method}
After mapping a given current time-step interval to the reference interval $(-1,1)$, the DG time-stepping method leads to the discrete problem of finding $u\in \Vp$ such that
\begin{equation}\label{eq:dg}
\begin{aligned}
\cB(u,v) &= \cF(v) 	& &\quad\forall\,v\in \Vp,
\end{aligned}
\end{equation}
where $\cF$ is a bounded linear functional on $\Vp$, and the bilinear form $\cB\colon \Vp \times \Vp \tends \R $ is defined by
\begin{equation}\label{eq:B_def}
\cB(u,v) \coloneqq \int_{-1}^{1} \left( u^\prime,v\right)_M \dd s + \left( u(-1),v(-1) \right)_M + \frac{\tau}{2} \int_{-1}^1\left(u,v\right)_A\dd s  ,
\end{equation}
where $\tau$ denotes the time-step size and $p$ denotes the polynomial degree of the approximation.
In practice, the time-step size, the polynomial degree, or even the matrices $M$ and $A$ may vary between time-steps, although on any given time-step the linear system has the general form of \eqref{eq:dg}.
Given a basis $\{\phi_j\}_{j=0}^p$ for $\Vp$, the problem \eqref{eq:dg} can be represented by a linear system
\begin{equation}\label{eq:B_system}
	\bo B \bo u = \bo f,
\end{equation}
where $\bo B$ is a $(p+1)\times (p+1)$ block matrix, where $\bo u \in \bv^{p+1}$ is the vector of coefficients of the expansion of $u$, and where $\bo f = (\bp f_0,\cdots, \bp f_p)$ with each $\bp f_j \in \bv^*$, $j=0,\dots,p$, being the restriction of $\cF$ to $\Span{\phi_j}\otimes \bv$. Therefore, the system \eqref{eq:B_system} has the block structure shown in~\eqref{eq:dg_linear_system}. 

\begin{remark}\label{rem:B_challenges}
In order to motivate our approach to preconditioning, it will be helpful to bear in mind the following point concerning the structure of the bilinear form $\cB$.
It is well-known that the bilinear form $\cB$ enjoys the following coercivity property
\begin{equation}\label{eq:coercivity_B}
\begin{aligned}
\cB(v,v) = \frac{1}{2} \norm{v(1)}_M^2 + \frac{1}{2} \norm{v(-1)}_M^2 + \frac{\tau}{2} \int_{-1}^1\norm{v}_A^2\,\dd s & & & \forall\,v\in \Vp,
\end{aligned}	
\end{equation}
which enables us to deduce the well-posedness of \eqref{eq:dg}. Unfortunately, the norm defined by the right-hand side of \eqref{eq:coercivity_B} does not include the time derivative, and thus coercivity and boundedness of $\calB$ in this norm can only be obtained from an inverse inequality for the finite dimensional space $\Vp$, at the expense of introducing constants that are not robust with respect to $\tau$ and $p$.
This point suggests that norm preconditioners for $\bo B$ that are based on the right-hand side of \eqref{eq:coercivity_B} are unlikely to be robust with respect to the discretization parameters.
\end{remark}

As we shall see below, our approach is based on the inf-sup stability of the bilinear form $\cB$, which provides a sharper analysis of the structure of the problem than the coercivity result of \eqref{eq:coercivity_B}. Some of our main tools are the reconstruction operator defined in section~\ref{sec:reconstruction}, and a corresponding suitable negative norm along with a key spectral equivalence result, which we recall in section~\ref{sec:negative_norm}. 

\subsection{Reconstruction operator}\label{sec:reconstruction}
We introduce the reconstruction operator $\calI \colon \calP_p \tends \calP_{p+1}$, defined by
\begin{equation}\label{eq:calI_definition}
\begin{aligned}
\calI v \coloneqq v - v(-1) \frac{(-1)^{p}\left(L_p-L_{p+1}\right)}{2} & & &\forall \,v\in \calP_{p}.
\end{aligned}
\end{equation}
As noted above, $\calI$ naturally extends to an operator from $\Vp$ to $\calV_{p+1}$. 
As explained in Remark~\ref{rem:radau_interpolant} below, $\calI$ is the reconstruction operator commonly used in a posteriori error analysis~\cite[]{Makridakis2006}.
The key benefit of the reconstruction operator $\calI$ for our purposes is that we  may express the bilinear form $\cB$ in the following equivalent form
\begin{equation}\label{eq:b_form_equiv}
\begin{aligned}
\calB(u,v) = \int_{-1}^1 \big((\calI u)^\prime , v\big)_M + \frac{\tau}{2} 	\left(u,v\right)_A \dd s 
& & &\forall\,u,\,v\in \Vp.
\end{aligned} 
\end{equation}
We emphasize here that $(\calI u)^\prime \in \Vp$ since $\calI u \in \calV_{p+1}$ for $u\in \Vp$.
To see how \eqref{eq:b_form_equiv} is obtained from \eqref{eq:calI_definition}, we note that the properties of the Legendre polynomials imply that, for any $v\in \Vp$,
\begin{equation}\label{eq:calI_properties}
\begin{aligned}
\calI v(1) = v(1), & & & \calI v(-1) = 0, & & &
\int_{-1}^1 (\calI v, w)_M\,\dd s = \int_{-1}^1 (v, w)_M\,\dd s \quad \forall \,w\in \calV_{p-1}.
\end{aligned}
\end{equation}
Substituting $w$ for $w^\prime$, which belongs to $\calV_{p-1}$ whenever $w\in \Vp$, in \eqref{eq:calI_properties} and using integration by parts shows that
\begin{equation}\label{eq:calI_properties_2}
	\int_{-1}^1 \left( (\calI v)^\prime,  w \right)_M\,\dd s = \int_{-1}^1 ( v^\prime, w)_M\,\dd s + \left( v(-1), w(-1) \right)_M \quad \forall\,w\in \Vp.
\end{equation}
The equivalent form of $\cB$ given in~\eqref{eq:b_form_equiv} then follows from \eqref{eq:calI_properties_2}.

\begin{remark}\label{rem:radau_interpolant}
The operator $\calI$ is the reconstruction operator commonly used in the a~posteriori error analysis of the DG time-stepping method \cite[]{Makridakis2006}. This operator is often defined by the interpolation conditions $\calI v(s_k) = v(s_k)$ for all $k=1,\dots,p+1$, where $-1<s_k \leq s_{p+1}=1$ are the Gauss--Radau quadrature points, in addition to a further condition, chosen here as $\calI v(-1)=0$. It turns out that the definition given above in \eqref{eq:calI_definition} and this interpolatory definition are equivalent, since the Gauss--Radau points are the roots of the polynomial $(1-s)P_{p}^{(1,0)}=L_p-L_{p+1}$, see \cite[]{Gautschi1997} and \cite[eq.~8.961.5]{TOISP}.  We note here that a straightforward consequence of the above properties is that, for any $v\in\calP_{p}$, we have $v\equiv 0$ in $(-1,1)$ if and only if $(\calI v)^\prime \equiv 0$ in $(-1,1)$. We also note that we will not need the Gauss--Radau points for the implementation of the preconditioners in this work.
\end{remark}

\subsection{Negative norms and a result of Pearson and Wathen}\label{sec:negative_norm}
In addition to the norms $\norm{\cdot}_M$ and $\norm{\cdot}_{A}$, we will also use the negative norm $\norm{\cdot}_{MA^{-1}M}$ defined by
\begin{equation}\label{eq:negative_norm}
\begin{aligned} 
\norm{\bp v}_{MA^{-1}M} \coloneqq \sup_{\bp w \in \bv\setminus\{0\}} \frac{ (\bp v,\bp w)_M }{\norm{\bp w}_A}  && &\forall\,\bp v \in \bv .
\end{aligned}
\end{equation}
The negative norm $\norm{\cdot}_{MA^{-1}M}$ can be equivalently characterized by the identity
\begin{equation}\label{eq:A_inv_norm_identity}
\begin{aligned}
\norm{\bp v}_{MA^{-1}M}^2  = \bp v^\top M A^{-1} M \bp v & & & \forall \,\bp v\in \bv,
\end{aligned}
\end{equation}
where $\bp v$ is identified with its vector representation in the basis $\{\bp e_i\}_{i=1}^{\Dim \bv}$. Indeed, \eqref{eq:A_inv_norm_identity} follows from the upper bound $\norm{\bp v}_{MA^{-1}M}^2 \geq \bp v^\top M A^{-1} M \bp v$, which is obtained by choosing $\bp w = A^{-1} M \bp v$ in \eqref{eq:negative_norm}, and from the lower bound $\norm{\bp v}_{MA^{-1}M}^2 \leq \bp v^\top M A^{-1} M\bp v$, which is obtained by applying the Cauchy--Schwarz inequality as follows: $\bp w^\top M \bp v = \bp w^\top A \left( A^{-1} M \bp v\right) \leq \sqrt{ \bp w^\top A \,\bp w } \;\sqrt{\bp v^\top M A^{-1} M \bp v }$, where we have simplified $( A^{-1} M \bp v)^\top A (A^{-1} M \bp v) = \bp v^\top M A^{-1} M \bp v$.
The identity \eqref{eq:A_inv_norm_identity} thus shows that the norm $\norm{\cdot}_{MA^{-1}M}$ is in fact induced by an inner product $\left( \cdot , \cdot \right)_{MA^{-1}M}$ represented by the matrix $M A^{-1} M$.

The following result due to J.~W.~Pearson and A.~J.~Wathen~\cite[]{Pearson2012} will play a key part in the construction of our preconditioners.

\begin{lemma}\label{lem:pearson}
Let $A$ and $M$ be arbitrary symmetric positive definite matrices and let $\mu\geq 0$ be a nonnegative real number. Then we have
\begin{equation}\label{eq:pearson}
\begin{aligned}
	\frac{1}{2} \leq \frac{ \bp v^\top \left( M A^{-1} M + \mu \,A \right) \bp v}{ \bp v^\top \left( M + \sqrt{\mu} \,A \right) A^{-1} \left( M + \sqrt{\mu}\, A \right) \bp v}\leq 1 & & &\forall\, \bp v \in \bv\setminus\{0\}.
\end{aligned}
\end{equation}
\end{lemma}
\begin{proof}
For the original proof of this result, see \cite[Thm~4]{Pearson2012}. We provide here an alternative proof which highlights the negative norm structure of the matrix $M A^{-1} M$, as given in \eqref{eq:negative_norm}. First, the upper bound of \eqref{eq:pearson} is immediate from
\[
\norm{\bp v}_{MA^{-1}M}^2 + \mu \norm{\bp v}_{A}^2 \leq \norm{\bp v}_{MA^{-1}M}^2 + 2 \sqrt{\mu} \norm{\bp v}_M^2 + \mu \norm{\bp v}_{A}^2.
\]
Now, for any $\bp v\in \bv$, we have $\norm{\bp v}_M^2 \leq \norm{\bp v}_{MA^{-1}M} \norm{\bp v}_A$ by \eqref{eq:negative_norm}.
It follows from the inequality $ ab\leq \frac{1}{2} a^2+\frac{1}{2} b^2$ for any $a$, $b\in \R$ that
\begin{equation}\label{eq:pearson_1}
\begin{split}
 \bp v^\top \left( M + \sqrt{\mu} \,A \right) A^{-1} \left( M + \sqrt{\mu}\, A \right) \bp v
& = \norm{\bp v}_{MA^{-1}M}^2 + 2\sqrt{\mu}\norm{\bp v}_M^2 + \mu \norm{\bp v}_{A}^2 \\
& \leq 2 \norm{\bp v}_{MA^{-1}M}^2 + 2 \mu \norm{\bp v}_{A}^2= 2\; \bp v^\top \left( M A^{-1} M + \mu\, A \right) \bp v.
\end{split}
\end{equation}
Note that \eqref{eq:pearson_1} is precisely the lower bound of \eqref{eq:pearson}.
\end{proof}

\section{Preconditioners}\label{sec:preconditioners}

\subsection{Left preconditioner}\label{sec:left_precond}
In this section, we construct a left preconditioner that will transform the linear system~\eqref{eq:B_system} to a symmetric positive definite system that represents a discrete parabolic energy norm. Our left preconditioner is defined simply in terms of a substitution for the test function $v$ appearing in the bilinear form $\calB(u,v)$. 
We start by defining the operator $P\colon \Vp\tends \Vp$ by
\begin{equation}
P v \coloneqq A^{-1} M (\calI v)^\prime + \frac{\tau}{2} v,
\end{equation}
where we recall the definition of the reconstruction operator $\calI$ from \eqref{eq:calI_definition} and its natural to extension to $\Vp$ as explained in section~\ref{sec:approximation_space}. 
The fact that $P v \in \Vp$ for any $v\in \Vp$ implies that the solution $u$ of \eqref{eq:dg} also solves
\begin{equation}\label{eq:dg_mod}
\begin{aligned}
\calL(u,v) & = \calG(v) & &\quad\forall\,v\in \Vp,
\end{aligned}
\end{equation}
where the bilinear form $\calL$ and linear functional $\calG$ are obtained by substituting the test function $Pv$ in place of $v$:
\begin{equation}\label{eq:L_def}
\begin{aligned}
&	\calL(u,v) \coloneqq \cB(u,P v), & & \calG(v) \coloneqq\cF(Pv) & &\forall\,u,\,v\in \Vp.
\end{aligned}
\end{equation}
We note that the operator $P$ can be viewed as defining a left preconditioner for the linear system \eqref{eq:B_system} that represents \eqref{eq:dg}. Indeed, let $\bo P$ denote the matrix representation of the linear operator $P$ in the basis $\{\phi_j\}_{j=0}^p$. Then, for any functions $u$ and $v \in \Vp$, and their respective vector representations $\bo u$ and $\bo v \in \bv^{p+1}$, we have $\cB(u,P v) = \bo v^\top \bo P^\top \bo B \bo u$ and $\cF(Pv) = \bo v^\top \bo P^\top \bo f$. Therefore, the linear system \eqref{eq:dg_mod} is equivalent to 
\begin{equation}\label{eq:lin_alg_mod}
\bo L \bo u = \bo g,
\end{equation}
with $\bo L = \bo P^\top \bo B$ and $\bo g = \bo P^\top \bo f$ denoting respectively the left-preconditioned matrix and right-hand side.

As the following theorem shows, $\calL$ represents a discrete version of the natural energy norm for parabolic problems. Indeed, in applications to second-order parabolic PDEs, $\calL$ is comparable to the inner product of $L^2(H^1)\cap H^1(H^{-1})$, which is the natural solution space of the continuous problem~\cite[]{Wloka1987}.

\begin{theorem}\label{thm:L_form}
Let the bilinear form $\calL\colon \Vp\times \Vp\tends \R$ and linear functional $\calG\colon \Vp\tends \R$ be defined by~\eqref{eq:L_def}. Then, for any functions $u$ and $v\in \Vp$, we have the identity
\begin{equation}\label{eq:L_form}
\begin{split}
\calL(u,v) &= \int_{-1}^1 \big( (\calI u)^\prime, (\calI v)^\prime\big)_{MA^{-1}M} \,\dd s + \frac{\tau^2}{4}\int_{-1}^1 (u,v)_A\,\dd s \\
& \qquad+ \frac{\tau}{2} \left(u(1),v(1)\right)_M +\frac{\tau}{2}\left(u(-1),v(-1)\right)_M  
\end{split}
\end{equation}
Therefore, $\calL$ is symmetric and positive definite on $\Vp$, and for any symmetric positive definite matrices $A$ and $M$, any $p\geq 0$ and any $\tau>0$, we have 
\begin{equation}\label{eq:calL_calA_bounds}
\begin{aligned}
\calL(v,v) \geq \norm{v}_{\calD}^2, & & & \abs{\calL(v,w)}\leq 2 \norm{v}_{\calD}\norm{w}_{\calD} & & &\forall\,v,\,w\in \Vp,
\end{aligned}
\end{equation}
where $\norm{\cdot}_{\calD}\coloneqq \sqrt{\calD(\cdot,\cdot)}$ is the norm induced by the auxiliary bilinear form $\calD$ defined by:
\begin{equation}\label{eq:calD_def}
	\calD(u,v) \coloneqq \int_{-1}^1 \big( (\calI u)^\prime, (\calI v)^\prime)\big)_{MA^{-1}M}\, \dd s + \frac{\tau^2}{4}\int_{-1}^1 (u,v)_A\dd s.
\end{equation}
Thus, the function $u\in \Vp$ solves \eqref{eq:dg} if and only if $u$ solves \eqref{eq:dg_mod}.
\end{theorem}

\begin{proof}
First, it is straightforward to obtain the following identities which hold for any $u,\,v \in \Vp$:
\begin{equation}\label{eq:inner_product_identities}
\begin{aligned}
\big((\calI u)^\prime,A^{-1}M (\calI v)^\prime\big)_M = ((\calI u)^\prime , (\calI v )^\prime)_{MA^{-1}M},  & & &
\big(u,A^{-1}M (\calI v)^\prime\big)_A   = (u,(\calI v)^\prime)_M.
\end{aligned}
\end{equation}
Next, we use the identities in \eqref{eq:inner_product_identities} to simplify the different terms in $\calL(u,v)=\calB(u,Pv)$, and eventually we obtain
\begin{equation}\label{eq:L_form_2}
\begin{split}
\calL(u,v) & = \int_{-1}^1 \big( (\calI u)^\prime, (\calI v)^\prime)\big)_{MA^{-1}M}\, \dd s + \frac{\tau^2}{4}\int_{-1}^1 (u,v)_A\,\dd s \\ &\qquad + \frac{\tau}{2} \int_{-1}^1 \big( (\calI u)^\prime , v \big)_M +  \big( u  , (\calI v)^\prime \big)_M\, \dd s .
\end{split}
\end{equation}
By \eqref{eq:calI_properties_2}, we have
\[
\begin{split}
	 \int_{-1}^1 \big( (\calI u)^\prime , v \big)_M +  \big( u  , (\calI v)^\prime \big)_M \,\dd s & =  \int_{-1}^1 \frac{\dd}{\dd s}(u,v)_M\,\dd s +   2 \left(u(-1),v(-1)\right)_M \\
&= \left(u(1),v(1)\right)_M+\left(u(-1),v(-1)\right)_M,
\end{split}
\]
which, after substitution of the last terms in \eqref{eq:L_form_2}, implies the equivalent form for $\calL$ given in~\eqref{eq:L_form}.
Next, to show \eqref{eq:calL_calA_bounds}, we note that first that the lower bound $\calL(v,v)\geq \norm{v}_{\calD}^2$ is immediate, whereas the upper bound $\abs{\calL(v,w)}\leq 2 \norm{v}_{\calD}\norm{w}_{\calD}$ follows from the application of the Cauchy--Schwarz inequality to \eqref{eq:L_form_2}.
Finally, it follows that the problem \eqref{eq:dg_mod} has a unique solution, momentarily denoted $\tilde{u}\in \Vp$.  Moreover, it is well-known that \eqref{eq:dg} has a unique solution $u\in \Vp$. Since the definition of $\calL$ and $\calG$ in~\eqref{eq:L_def} shows that $u$ is a solution of \eqref{eq:dg_mod}, we deduce from uniqueness that $\tilde{u}=u$. Therefore, the problems \eqref{eq:dg} and \eqref{eq:dg_mod} are equivalent.
\end{proof}

\begin{remark}\label{rem:Kreuzer_infsup}
Theorem~\ref{thm:L_form} can be viewed as part of the inf-sup analysis of the DG time-stepping method: defining the energy norm $\norm{\cdot}_{\calL}\coloneqq \sqrt{\calL (\cdot,\cdot)}$ and the norm $\norm{\cdot}_{\calX}\coloneqq\sqrt{\int_{-1}^1 \norm{\cdot}^2_A\dd s}$, we have
\begin{equation}\label{eq:inf_sup}
\begin{aligned}
\norm{u}_{\calL} = \sup_{v\in \Vp\setminus\{0\}} \frac{\calB(u,v)}{\norm{v}_{\calX}}, & & & \abs{\calB(u,v)}\leq \norm{u}_{\calL}\norm{v}_{\calX} & & & \forall\,u,\,v\in \Vp,
\end{aligned}
\end{equation}
where the first equality is attained by choosing the optimal test function $v=Pu$, since $\norm{Pu}_{\calX}=\norm{u}_{\calL}$. This observation highlights the fact that the operator $P$ represents the operator that gives the optimal test function in the inf-sup analysis of the DG time-stepping method. It is in this sense that our preconditioning strategy is directly motivated by the inf-sup theory of the DG time-stepping method.
\end{remark}

\subsection{Spectrally equivalent norm preconditioner}\label{sec:spect_equiv}
As shown by Theorem~\ref{thm:L_form}, the bilinear form $\calL$ is symmetric and positive definite. Therefore, the linear system~\eqref{eq:lin_alg_mod} can be solved iteratively by the preconditioned conjugate gradient~(PCG) algorithm \cite[]{Hestenes1952,MalekStrakos2015,Wathen2015}. In this section, we construct a spectrally equivalent and easily applicable preconditioner for $\bo L$, thereby leading to the robust and fast convergence of the PCG algorithm.
Recall that the bilinear form $\calL$ is spectrally equivalent to the bilinear form $\calD$ defined by \eqref{eq:calD_def}.
The first step in our construction of a preconditioner is to choose an advantageous temporal basis for $\Vp$ that block-diagonalizes the matrix $\bo D$ that represents $\calD$. In a second step, we construct an easily applicable preconditioner, denoted by $\bo H$, under this basis, and we apply Lemma~\ref{lem:pearson} to show robust spectral equivalence between $\bo H$ and $\bo L$.
 
\subsubsection{Definition of the basis} 
It follows from Remark~\ref{rem:radau_interpolant} that the symmetric bilinear form $(u,v)\tends \int_{-1}^1 (\calI u)^\prime (\calI v)^\prime\,\dd s$ is positive definite and thus defines an inner-product on $\calP_p$.
Therefore, there exists a set of linearly independent polynomial eigenfunctions $\{\varphi_j\}_{j=0}^p \subset \calP_p$ and corresponding real positive eigenvalues $\{\lambda_j\}_{j=0}^p\subset\R_{>0}$ such that
\begin{equation}\label{eq:eigenfunctions}
\begin{aligned}
\lambda_j \int_{-1}^1  (\calI \varphi_j)^\prime \, (\calI v)^\prime\, \dd s  = \int_{-1}^1 \varphi_j \,v\,\dd s  & & &\quad \forall\, v\in \calP_p,\;j=0,\dots,p.
\end{aligned}
\end{equation}
We note that the eigenvalues and eigenfunctions generally depend on $p$, and are not necessarily hierarchical.
The eigenfunctions are chosen to be orthonormalized:
\begin{equation}\label{eq:eigenfunctions_normalized}
\begin{aligned}
\int_{-1}^1  (\calI \varphi_j)^\prime \, (\calI \varphi_k)^\prime\, \dd s = \delta_{jk} & & &\forall \, 0\leq j,\,k\leq p.
\end{aligned}
\end{equation}
We provide a practical method for computing this basis along with the corresponding eigenvalues in section~\ref{sec:comp_eigenfunctions}.
The following result characterizes the distribution of the eigenvalues.
\begin{theorem}\label{thm:eigenvalue_distribution}
Let $p\geq 0$ be a nonnegative integer, and let $\lambda_0\geq \dots \geq \lambda_p$ be the eigenvalues of the generalized eigenvalue problem~\eqref{eq:eigenfunctions}. Then, there exists a positive constant $C$, independent of $j$ and $p$, such that
\begin{equation}\label{eq:eigenvalue_bound}
\begin{aligned}
\lambda_{j}  \leq C (j+1)^{-2}	& & & \forall\,  j=0,\dots, p,\; \forall\,p\geq 0.
\end{aligned}
\end{equation}
Moreover, there exists a constant $c$, independent of $p$, such that
\begin{equation}\label{eq:eigenvalue_lower_bound}
\lambda_0\geq \dots \geq \lambda_p \geq c \left(p+1\right)^{-4}.
\end{equation}
\end{theorem} 
We leave the proof of Theorem~\ref{thm:eigenvalue_distribution} to Appendix~\ref{sec:eigenvalues} since it is based on some results presented in the later section~\ref{sec:comp_eigenfunctions}. Figure~\ref{fig:eigenvalues} shows that the orders of the bounds of Theorem~\ref{thm:eigenvalue_distribution} are sharp. The eigenfunctions for $p=4$ are shown in Figure~\ref{fig:eigenfunctions}.

\subsubsection{Construction of the norm preconditioner} It is advantageous to use the basis $\{\varphi_j\}_{j=0}^p$ in computations since in this basis, the dominant terms in the bilinear form $\calL$, namely those belonging to the bilinear form $\calD$, are represented by a block-diagonal matrix. Indeed, let $\bo D$ denote the matrix representation of the bilinear form $\calD$, defined in \eqref{eq:calD_def}, under the basis $\{\varphi_j\}_{j=0}^p$ of $\Vp$, where we recall the terminology from Remark~\ref{rem:basis}.
Then, in this basis, the matrix $\bo D$ has a simple block-diagonal structure:\begin{equation}
\bo D = \Diag\left\{ M A^{-1} M + \frac{\tau^2\, \lambda_j}{4} A \right\}_{j=0}^{p}.
\end{equation}
In other words, $\bo D$ is block-diagonal, with blocks $M A^{-1} M + \lambda_j \frac{\tau^2}{4} A$ for $j=0,\dots,p$.
Keeping in mind Lemma~\ref{lem:pearson}, it is then natural to define the norm preconditioner $\bo H$ by
\begin{equation}\label{eq:H_def}
\bo H \coloneqq \Diag\left\{ \left(M+\frac{\tau\,\sqrt{\lambda_j}}{2} A\right)\, A^{-1} \,\left(M+\frac{\tau\,\sqrt{\lambda_j}}{2} A\right) \right\}_{j=0}^p.
\end{equation}
The matrix $\bo H$ is symmetric positive definite, and its inverse is  trivially given by
\[
\bo H^{-1} = \Diag\left\{ \left(M+\frac{\tau\,\sqrt{\lambda_j}}{2} A\right)^{-1}\, A \, \left(M+\frac{\tau\,\sqrt{\lambda_j}}{2} A\right)^{-1}\right\}_{j=0}^p.
\]
We propose to use $\bo H$ as a preconditioner for the PCG algorithm applied to \eqref{eq:lin_alg_mod}.
Each PCG iteration requires the application of $\bo H^{-1}$, which comprises two applications of a solver for a weighted backward Euler step and one multiplication by $A$ per block.

We now give the central result of this work, which shows that $\bo L$ and $\bo H$ are spectrally equivalent with fully robust bounds.

\begin{theorem}\label{thm:L_H_equivalence}
Let $p\geq 0$, and let $\{ \varphi_j \}_{j=0}^p\subset \calP_p$ and $\{\lambda_j\}_{j=0}^p\subset \R_{>0}$ be the eigenfunctions and eigenvalues of \eqref{eq:eigenfunctions}. Let $\bo L$ be the matrix representation of $\calL$ in the basis $\{\varphi_j\}_{j=0}^p$. Then, for any symmetric positive definite matrices $A$ and $M$, any $\tau>0$ and any $p\geq 0$, we have
\begin{equation}\label{eq:L_H_equivalence}
\begin{aligned}
\frac{1}{2} \leq \frac{ \bo v^\top \bo L\, \bo v }{\bo v^\top \bo H \,\bo v} \leq 2  \quad \forall\, \bo v \in \bv^{p+1}\setminus\{0\}.
\end{aligned}	
\end{equation}
\end{theorem}
\begin{proof}
Theorem~\ref{thm:L_form}, in particular \eqref{eq:calL_calA_bounds}, shows that
\begin{equation}\label{eq:L_A_equivalence}
\begin{aligned}
1\leq \frac{\bo v^\top \bo L\, \bo v }{\bo v^\top \bo D \,\bo v} \leq 2  \quad \forall\,\bo v\in \bv^{p+1}\setminus\{0\}.
\end{aligned}
\end{equation}
Lemma~\ref{lem:pearson}, applied block-wise with $\mu = \lambda_j \tau^2/4$, implies that 
\begin{equation}\label{eq:A_H_equivalence}
\begin{aligned}
\frac{1}{2}\leq \frac{\bo v^\top \bo D\, \bo v }{\bo v^\top \bo H \,\bo v} \leq 1  \quad \forall\,\bo v\in \bv^{p+1}\setminus\{0\}.
\end{aligned}
\end{equation}
Therefore, \eqref{eq:L_A_equivalence} and \eqref{eq:A_H_equivalence} imply \eqref{eq:L_H_equivalence}.
\end{proof}

Theorem~\ref{thm:L_H_equivalence} immediately implies that  the condition number $\kappa$ of the preconditioned system, defined as the ratio of extremal eigenvalues of $\bo H^{-1} \bo L$, satisfies
\begin{equation}\label{eq:condition_number}
\kappa \leq 4.	
\end{equation}
Therefore, we may expect very fast convergence from the PCG algorithm for $\bo L \bo u = \bo g$ using the preconditioner $\bo H$. Indeed, for an initial guess $u_0\in \Vp$, the iterates $\{u_k\}_{k\geq 0}$ of the PCG algorithm satisfy the well-known bound \cite[]{Wathen2015}
\begin{equation}\label{eq:CG_bound}
\begin{aligned}
\frac{\norm{u - u_k}_{\calL}}{\norm{u-u_0}_{\calL}} \leq 2 \left( \frac{\sqrt{\kappa}-1}{\sqrt{\kappa}+1}\right)^k \leq \frac{2}{3^k} & & &\forall\,k\geq 0,
\end{aligned}
\end{equation}
where we recall the energy norm $\norm{\cdot}_{\calL} = \sqrt{\calL(\cdot,\cdot)}$.
The fact that the condition number $\kappa \leq 4$ shows that the left preconditioner $\bo P^\top$ and the norm preconditioner $\bo H$ are very effective at preconditioning the original system matrix $\bo B$, despite $\bo B$ being nonsymmetric and poorly conditioned.
This is a key aspect of the efficiency of the proposed preconditioners.
Furthermore, we highlight that \eqref{eq:condition_number} and \eqref{eq:CG_bound} are valid for any time-step size~$\tau$, any polynomial degree $p$, and any symmetric positive definite matrices $M$ and $A$. Thus the preconditioners are fully robust with respect to all discretization and problem parameters.

\begin{remark}
The fact that the energy norm $\norm{\cdot}_{\calL}$ appears in the bound \eqref{eq:CG_bound} is advantageous in practice. Since the endpoint value of the solution at $s=1$ serves as initial datum for the next time-step, it is beneficial to be able to control the accuracy to which it is computed. For example, \eqref{eq:L_form} shows that the energy norm controls the value at $s=1$ in the $M$-norm, which is the natural norm for this quantity of interest.
We note that the factor of $\tau$ appearing alongside $\norm{v(1)}_M^2$ in the energy norm can be scaled out as it appears in both the numerator and denominator of \eqref{eq:CG_bound}. Therefore \eqref{eq:CG_bound} is a guaranteed  convergence rate for this quantity of interest in the physically relevant norm.
\end{remark}
 
\begin{figure}
\begin{center}
\includegraphics{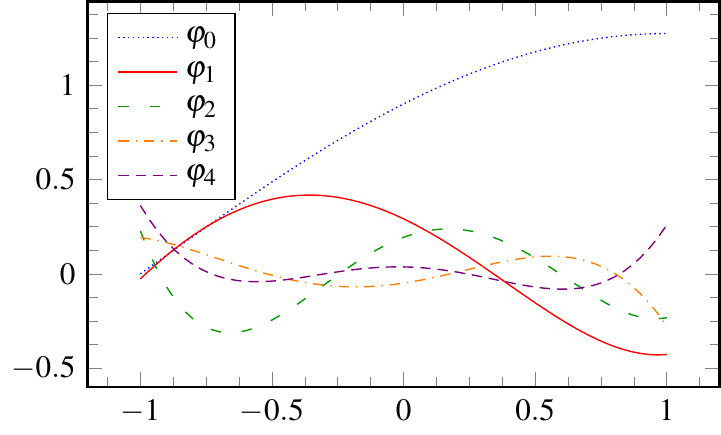}
\caption{Eigenfunctions $\{\varphi_j\}_{j=0}^p$ defined by \eqref{eq:eigenfunctions} and \eqref{eq:eigenfunctions_normalized}, computed for $p=4$ by the method of section~{\upshape\ref{sec:comp_eigenfunctions}} and ordered by decreasing eigenvalue $\lambda_j\geq \lambda_{j+1}$. Note that the $j$-th eigenfunction $\varphi_j$ need not be of degree at most~$j$. Furthermore, the set of eigenfunctions generally depends on $p$, although they are independent of $A$, $M$ and $\tau$.}
\label{fig:eigenfunctions}
\includegraphics{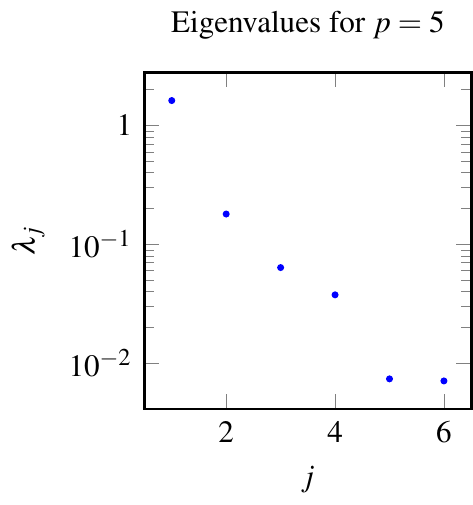}
\includegraphics{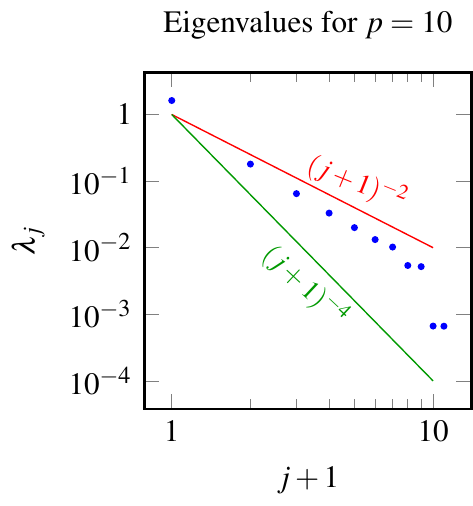}
\includegraphics{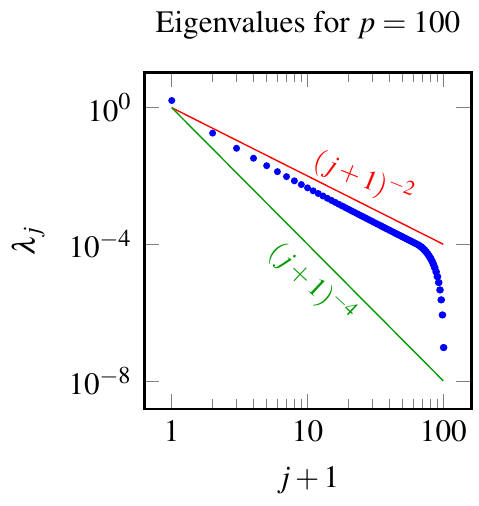}
\includegraphics{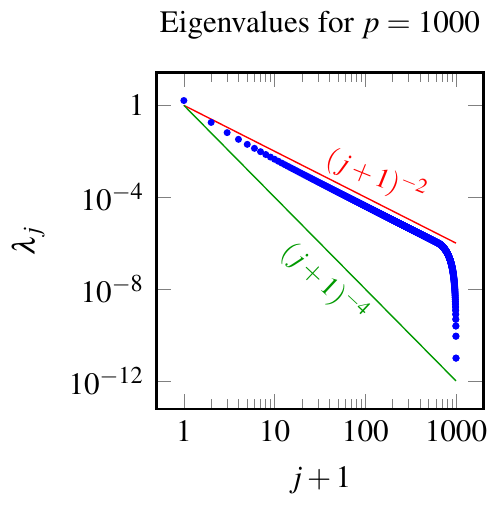}
\caption{Eigenvalues $\{\lambda_j\}_{j=0}^p$ computed by the method of section~{\upshape\ref{sec:comp_eigenfunctions}} and arranged in decreasing order, for $p=5$, $10$, $100$ and $100$. All plots are on a logarithmic scale, except for $p=5$ which is given on a semilogarithmic scale. It appears that the bounds of Theorem~{\upshape\ref{thm:eigenvalue_distribution}} are sharp.}
\label{fig:eigenvalues}
\end{center}
\end{figure}

\section{Implementation}\label{sec:implementation}

\subsection{Computation of eigenfunctions and eigenvalues}\label{sec:comp_eigenfunctions}
The algorithm requires the computation of the eigenfunctions $\{\varphi_j\}_{j=0}^p$ and eigenvalues $\{\lambda_j\}_{j=0}^p$ defined in \eqref{eq:eigenfunctions}, which involves solving a symmetric positive definite eigenvalue problem of dimension~$p+1$. This poses little difficulty, as in practice $p$ is usually small, especially in comparison to $\Dim\bv$. Since this step only depends on $p$ and is fully independent of $A$, $M$ and $\tau$, it can be pre-computed to very high accuracy.
We now show how to assemble this eigenvalue problem in a form where it can be solved numerically. Since the cases $p=0$ and $p=1$ can be easily computed by hand, we present a general method for $p\geq 2$.
In the following, we will use the following identity for the Legendre polynomials $\{L_k\}_{k=0}^p$, see \cite[Sec.~8.914]{TOISP}:
\begin{equation}\label{eq:legendre_derivative_identity}
\begin{aligned}
L_k  = \frac{L_{k+1}^\prime - L_{k-1}^\prime}{2k+1}	 & & &\forall\, k\geq 0.
\end{aligned}
\end{equation}
\begin{lemma}\label{lem:psi_basis}
Let $p\geq 2$ be a nonnegative integer. Define the polynomials
\begin{equation}\label{eq:psi_definition}
\begin{aligned}
\psi_0 \coloneqq \frac{1}{\sqrt{2}}(L_1 + L_0) &&& \psi_p \coloneqq \frac{L_p-L_{p-1}}{\sqrt{4p+2}},\\
\psi_k \coloneqq \frac{L_{k+1}-L_{k-1}}{\sqrt{4k+2}}, & & & k=1,\dots,p-1.
\end{aligned}
\end{equation}
Then, there holds
\begin{equation}\label{eq:inverse_calI}
\begin{aligned}
(\calI \psi_k)^\prime  = \sqrt{k+\tfrac{1}{2}}\, L_k & & & \forall\,k=0,\dots,p.
\end{aligned}
\end{equation}
Therefore, we have, for each $0\leq k,\,j\leq p$,
\begin{equation}\label{eq:psi_orthogonality}
\int_{-1}^1 (\calI \psi_k)^\prime (\calI \psi_j)^\prime \,\dd s = \delta_{kj}.
\end{equation}
\end{lemma}
\begin{proof}
Consider the case $k=0$: $\psi_0 (-1) = 0$, and $L_1^\prime = L_0$. Therefore $\calI \psi_0 = \psi_0$, and thus $(\calI \psi_0)^\prime = L_0/\sqrt{2}$, which is \eqref{eq:inverse_calI} for $k=0$.
Now consider the general case $k=1,\dots,p-1$. Since $L_j(-1) = (-1)^j$ for all $j\geq 0$, we have $\psi_k(-1)=0$ and thus $\calI \psi_k = \psi_k$. Now, the identity \eqref{eq:legendre_derivative_identity} implies that
\begin{equation}\label{eq:inverse_calI_0}
(\calI \psi_k)^\prime = \frac{2k+1}{\sqrt{4k+2}} L_k = \sqrt{k+\tfrac{1}{2}}\,L_k,
\end{equation}
thus verifying \eqref{eq:inverse_calI} for $k=1,\dots,p-1$.
For the case $k=p$, we use the fact that $L_p(-1)-L_{p-1}(-1)=2(-1)^p$ to compute
\begin{equation}
\calI (L_p-L_{p-1}) = L_p - L_{p-1} - \frac{2(-1)^p(-1)^p}{2}(L_p-L_{p+1}) = L_{p+1} - L_{p-1}.
\end{equation}
Therefore, similarly to \eqref{eq:inverse_calI_0}, the identity \eqref{eq:inverse_calI} for $k=p$ follows from \eqref{eq:legendre_derivative_identity}.
\end{proof}

Define the matrix $Z\in \R^{(p+1)\times(p+1)}$ such that $\psi_{k}=\sum_{j=0}^p Z_{kj} L_j$ as in \eqref{eq:psi_definition}. Define the diagonal matrix $D \coloneqq \Diag\{ 2/(2j+1) \}_{j=0}^p$, which represents the $L^2$-inner product in the Legendre polynomial basis. 
Define the matrix $T\in \R^{(p+1)\times (p+1)}$ by 
\begin{equation}\label{eq:T_def}
\begin{aligned}
T_{kj} \coloneqq \int_{-1}^1 \psi_k\,\psi_j\,\dd s, & & & k,\,j=0,\dots,p.
\end{aligned}
\end{equation}
Note that $T = Z D Z^\top$.
It follows from \eqref{eq:psi_orthogonality} that the eigenvalue problem \eqref{eq:eigenfunctions} can be expressed as: find $V \in \R^{(p+1)\times(p+1)}$ and $\{\lambda_j\}_{j=0}^p$ such that
\begin{equation}\label{eq:discrete_eigenvalue}
\begin{aligned}
T V =  V \Diag\{\lambda_j\}_{j=0}^p, & & & V V^\top = V^\top V = \Id.
\end{aligned}
\end{equation}
The eigenvalue problem \eqref{eq:discrete_eigenvalue} can thus be solved numerically by standard eigenvalue solvers.
The eigenfunctions $\{\varphi_j\}_{j=0}^p$ can then be recovered as
\begin{equation}\label{eq:eigenfunctions_V_def}
\begin{aligned}
\varphi_j(s) = \sum_{k=0}^p V_{kj} \,\psi_{k}(s) = \sum_{k=0}^p Q_{kj} \,L_k(s), & & & Q \coloneqq Z^\top V.
\end{aligned}
\end{equation}

\subsection{Implementation of preconditioners} In order to compute the action of the left-preconditioner $\bo P^\top$, we use the matrix representation of the mapping $v\mapsto (\calI v)^\prime$ in the basis given by $\{\varphi_j\}_{j=0}^p$. Thus we need to find the matrix $K$ such that $(\calI \varphi_k)^\prime = \sum_{j=0}^p K_{kj}\, \varphi_j$ for all $k=0,\dots,p$. We show below that this is easily computed.
Inverting \eqref{eq:legendre_derivative_identity} by induction yields\begin{equation}\label{eq:derivative_legendre}
\begin{aligned}
L_k^\prime = \sum_{j=0}^{k-1} \left(1-(-1)^{k-j}\right)\left(j+\tfrac{1}{2}\right) L_j & & & \forall \,k\geq 0.
\end{aligned}
\end{equation}
Therefore, we have
\begin{equation}\label{eq:calI_legendre}
\begin{aligned}
\left(\calI L_k\right)^\prime = \sum_{j=0}^p K^*_{kj} \,L_j & & & \forall\,k=0,\dots,p,
\end{aligned}
\end{equation}
where, after some calculation, it is found that the matrix $K^* \in \R^{(p+1)\times (p+1)}$ can be defined in terms of $\tilde{K}\in\R^{(p+1)\times (p+1)}$, $x\in \R^{p+1}$, and $y \in \R^{p+1}$ by
\begin{align*}
K^* \coloneqq \tilde{K} - x\,y^\top, & &  & \tilde{K}_{kj}\coloneqq\begin{cases}
\left(1-(-1)^{k-j}\right)\left(j+\tfrac{1}{2}\right) & \text{if }j\leq k,\\
0 & \text{otherwise},
\end{cases}\\
x_j \coloneqq (-1)^j, & & & y_{k} \coloneqq (-1)^{1-k}\left(k+\tfrac{1}{2}\right).
\end{align*}

\begin{lemma}
Let $p\geq 2$ be an integer, and let $\{\varphi_j\}_{j=0}^p$ be the eigenfunctions defined by \eqref{eq:eigenfunctions_V_def}. Then, for each $k=0,\dots,p$, there holds
\begin{equation}
\begin{aligned}
(\calI \varphi_k)^\prime = \sum_{j=0}^p K_{kj}\, \varphi_j, & & & K \coloneqq V^\top D^{-1/2} K^* D^{1/2} V,
\end{aligned}
\end{equation}
where $V$ is as in \eqref{eq:discrete_eigenvalue}, the matrix $K^*$ is as in \eqref{eq:calI_legendre}, and where $D = \Diag\{ 2/(2j+1) \}_{j=0}^p$.
\end{lemma}
\begin{proof}
The orthogonality of the matrix $V$ in \eqref{eq:discrete_eigenvalue} implies that $\psi_k = \sum_{m=0}^p V_{km} \varphi_m$.
Since the basis $\{\psi_k\}_{k=0}^p$ is orthonormal in the inner product of \eqref{eq:psi_orthogonality}, there holds
\[
\begin{split}
L_j &= \sum_{r=0}^p \left(\int_{-1}^1 (\calI L_j)^\prime (\calI \psi_r)^\prime\,\dd s \right) \psi_r = \sum_{r=0}^p \sum_{m=0}^p \left(K^*_{jm}  (r+\tfrac{1}{2})^{1/2} \int_{-1}^1 L_m \, L_r \,\dd s \right) \psi_r  \\
& = \sum_{r=0}^p K^*_{jr} \left(r+\tfrac{1}{2}\right)^{-1/2} \psi_r = \sum_{r=0}^p (K^* D^{1/2})_{jr} \psi_r = \sum_{m=0}^p (K^* D^{1/2} V)_{jm} \varphi_m,
\end{split}
\]
where we have used \eqref{eq:inverse_calI} and \eqref{eq:calI_legendre} in the first line.
Therefore, we compute
\[
(\calI \varphi_k)^\prime =\sum_{j=0}^p V_{jk} (\calI \psi_j)^\prime = \sum_{j=0}^p V_{jk} (j+\tfrac{1}{2})^{1/2} L_j 
= \sum_{j=0}^p \sum_{m=0}^p V_{jk} ( D^{-1/2} K^* D^{1/2} V)_{jm} \varphi_m,
\]
which completes the proof.
\end{proof}

We now show how the matrix $K$ is used in applying the preconditioner~$\bo P^\top$.
For $v = \sum_{k=0}^p \bp v_k \,\varphi_k$, there holds
\begin{gather*}
Pv = \sum_{j=0}^p \left( A^{-1} M \bp w_j + \frac{\tau}{2} \bp v_j \right) \varphi_j, \quad \bp w_j \coloneqq \sum_{k=0}^p K_{kj} \bp v_k, \\
\calF(Pv) = \sum_{j=0}^p \bp f_j^\top \left( A^{-1} M \bp w_j +\frac{\tau}{2} \bp v_j \right)=\sum_{k=0}^p \left( M A^{-1}\left( \sum_{j=0}^p K_{kj}\, \bp f_j\right) + \frac{\tau}{2} \bp f_k \right)^\top \bp v_k.
\end{gather*}
Therefore, $ \bo g \coloneqq \bo P^\top \bo f$, where $\bo g = (\bp g_0,\dots,\bp g_p)$, can be computed componentwise by
\begin{equation}
\begin{aligned}
\bp g_k = M A^{-1}\left( \sum_{j=0}^p K_{kj} \,\bp f_j\right) + \frac{\tau}{2} \bp f_k, & & & k=0,\dots,p.
\end{aligned}
\end{equation}
The action of $\bo P^\top$ requires the solution of $p+1$ independent systems with matrix $A$, which can be performed in parallel.
If we ignore communication costs, then the cost of computing $\bo g$ is independent of the polynomial degree $p$ on a parallel machine with sufficiently many computing nodes.

After application of the preconditioner $\bo P^\top$, the linear system has the form
\[
\bo L \bo u = \bo g,
\]
where $\bo L = \bo P^\top \bo B$, which can be solved by the PCG algorithm with preconditioner $\bo H$, as suggested in section~\ref{sec:preconditioners}. The PCG algorithm requires the action of the matrices $\bo L$  and $\bo H^{-1}$ at each iteration. There are two ways to implement the action of $\bo L$, the first being the application of $\bo B$ followed by $\bo P^\top$.
The downside of this approach is that it leads to greater communication costs on a distributed memory parallel machine where each node holds in memory only a few coefficients of $\bo u$: the matrices $\bo B$ and $\bo P^\top$ are generally block-dense, so all vector components must be exchanged between all computational nodes for both steps.
The second approach is to use \eqref{eq:L_form}, which shows that $\bo L$ can be expressed as a diagonal matrix plus a ``rank-two'' term: in the basis $\{\varphi_j\}_{j=0}^p$, we have
\begin{equation}
\bo L = \Diag\left\{ M A^{-1} M +\frac{\tau^2 \lambda_j}{4} A  \right\}_{j=0}^p + \frac{\tau}{2} \bo q_1 \bo q_1^\top \otimes M + \frac{\tau}{2} \bo q_{-1} \bo q_{-1}^\top \otimes M,
\end{equation}
where $\bo q_{\pm 1} \coloneqq \left(q_0(\pm 1),\dots,q_p(\pm 1)\right)$ are the vectors of endpoint values of the $\varphi_j$. Therefore, we may compute $\bo L \bo u$, with $\bo u = (\bp u_0,\dots, \bp u_p)$, as follows:
\begin{subequations}\label{eq:compute_L}
\begin{gather}
\bp w_j \coloneqq M \bp u_j,\quad j=0,\dots,p,  \qquad \bp z_{\pm 1} \coloneqq \sum_{j=0}^p \varphi_j(\pm 1) \bp w_j, \\
(\bo L \bo u)_j = M A^{-1} \bp w_j + \frac{\tau^2 \lambda_j}{4} A \bp u_j + \frac{\tau}{2} \varphi_j(1) \bp z_1 + \frac{\tau}{2}  \varphi_j(-1) \bp z_{-1}. 
\end{gather}
\end{subequations}
This approach requires the same number of matrix-vector products as the first approach outlined above, but the communication costs are greatly reduced. Indeed, it is seen that the above procedure requires the parallel computation of the $\bp w_j$, which are then gathered and reduced by one or two nodes tasked with computing $\bp z_{\pm 1}$. The results are then broadcast back to all nodes, after which all subsequent computations can be performed in parallel.

Finally, the application of $\bo H^{-1}$ is trivially parallel, as it requires only that each node solves two linear systems involving the matrix $M + \tau \sqrt{\lambda_j}/2 A$, and one application of $A$. Theorem~\ref{thm:eigenvalue_distribution} shows that the eigenvalues remain uniformly bounded from above, and that the smaller eigenvalues approach zero as $p$ increases. Therefore, for large $p$ and $j$, the matrix $M + \tau \sqrt{\lambda_j}/2 A$ becomes increasingly similar to $M$, implying that these systems will become increasingly easy to solve in many applications.
 
\section{Numerical experiments}\label{sec:num_exp}

\subsection{Condition numbers}\label{sec:num_exp_1}
In this experiment, we study the dependence of the condition numbers $\kappa$ of the preconditioned system $\bo H^{-1} \bo L$ on the problem and discretization parameters. In order to compute accurately the condition numbers, we consider first a one-dimensional problem.
Let $M$ and $A$ be the mass and stiffness matrices obtained by applying piecewise-linear continuous finite elements on a uniform subdivision of the domain $\Om=(0,1)$ into elements of size $h=2^{-k}$, $k=5,\dots,10$. We note that $M$ and $A$ thus depend on $h$, although this is left implicit in our notation. For this experiment, we use direct solvers to implement the action of $A^{-1}$ and $(M + \tau \sqrt{\lambda_j}/2 A)^{-1}$. 

Table~\ref{tab:tau_dependence} shows the condition numbers $\kappa$ as a function of $\tau$ for a wide range of parameters. For this experiment, we set $p=2$ and $h=2^{-5}$. It is found that for either very large or very small $\tau$, $\kappa\tends1$. This is easily explained by the fact that\begin{equation}\label{eq:tau_asymptotics}
\lim_{\tau\tends \infty} 	\frac{ \bo v^\top \bo L\, \bo v }{\bo v^\top \bo H \,\bo v} = \lim_{\tau\tends 0} \frac{ \bo v^\top \bo L\, \bo v }{\bo v^\top \bo H \,\bo v} = 1.
\end{equation}
Therefore the condition number $\kappa\tends 1$ as $\tau\tends 0 $ or $\tau \tends \infty$. Thus the maximal condition number is found for intermediate values of the time-step size $\tau$, although in all cases it satisfies the theoretical bound $\kappa\leq 4$ as shown in~\eqref{eq:condition_number}.
\begin{table}[b]
\begin{center}
\begin{tabular}{| c | c c c c c c c c |}
\hline  & & & & & & & & \vspace{-2.5ex} \\
$\tau$ &  $10^{-6}$ & $10^{-5}$ & $10^{-4}$ & $10^{-3}$ & $10^{-2}$ & $10^{-1}$ & $1$ & $10$\\ \hline 
$\kappa$ & $1.011$ & $1.103$ & $1.749$ & $2.031$ & $2.028$ & $2.019$ & $1.693$ & $1.089$ \\ \hline 
\end{tabular}\vspace{1ex}
\caption{Dependence of the condition number $\kappa$ of $\bo H^{-1}\bo L$ on a wide range of values for the time-step size $\tau$. Observe that for either very large or very small $\tau$, the condition number is asymptotically $1$, in agreement with \eqref{eq:tau_asymptotics}. This explains the observation that the maximum condition number is observed for intermediate values.}
\label{tab:tau_dependence}
\end{center}
\end{table}
\begin{table}
\begin{center}
\begin{tabular}{| c | c c c c c c |}
\hline  & & & & & &  \vspace{-2.5ex} \\
$ \kappa $ & $h=2^{-5}$ & $h=2^{-6}$ & $h=2^{-7}$ & $h=2^{-8}$ & $h=2^{-9}$ & $h=2^{-10}$ \\ \hline 
$p=1$ & 1.318 & 1.319 & 1.319 & 1.319 & 1.319 & 1.319\\
$p=2$ & 2.019 & 2.019 & 2.019 & 2.019 & 2.019 &2.019\\
$p=3$ & 2.243 & 2.243 & 2.243 & 2.243 & 2.243 & 2.243\\
$p=4$ & 2.353 & 2.353 & 2.353 & 2.353 & 2.353 & 2.353\\
$p=5$ & 2.416 & 2.417 & 2.417 & 2.417 & 2.417 & 2.417\\
$p=6$ & 2.493 & 2.493 & 2.493 & 2.493 & 2.493 & 2.493\\ \hline 
\end{tabular}
\vspace{1ex}
\caption{Dependence of the condition number $\kappa$ for $p=1,\dots,6$ and $h=2^{-5},\dots,2^{-10}$, with $\tau=0.1$. For fixed $p$, the condition number is insensitive to the mesh size $h$, i.e.\ variation of $M$ and $A$. The condition number slowly increases with $p$ towards its asymptotic value of approximately $2.686$, see Table~{\upshape\ref{tab:asymptotic_p}}.}
\label{tab:ph_dependence}
\end{center}
\end{table}
Table~\ref{tab:ph_dependence} shows that the condition number has little to no dependence on mesh refinement, i.e.\ variation of $M$ and $A$. As $p$ becomes very large, the condition number approaches an asymptotic value of around $2.686\leq 4$, as shown by Table~\ref{tab:asymptotic_p}; our goal in testing our preconditioners here with very high polynomial degrees is merely to ascertain the asymptotic behaviour of the condition number.
The proposed preconditioners are thus fully robust with respect to the parameters, in agreement with the theoretical bound $\kappa\leq 4$ from~\eqref{eq:condition_number}. Furthermore, this experiment suggests that the efficiency of the preconditioners can exceed theoretical expectations, as shown by condition numbers $\kappa \leq 2.686$ throughout these tests.

\begin{table}
\begin{center}
\begin{tabular}{| c |c c c c c c  |}
\hline  & & & & & &  \vspace{-2.5ex} \\
$p$ & 8 & 16 & 32 & 64 & 128 & 256 \\ \hline 
$\kappa$ & 2.558 & 2.643 & 2.674  & 2.684 & 2.686 &2.686 \\ \hline
\end{tabular}
\vspace{1ex}
\caption{Asymptotic behaviour of the condition number with respect to the polynomial degree $p=2^m$, $m=3,\dots,8$, with fixed $h=2^{-5}$ and $\tau=0.1$. The asymptotic value of $\kappa$ appears to be $\kappa \approx 2.686$. This suggests that the theoretical bound $\kappa \leq 4$ of \eqref{eq:condition_number} may not be quantitatively sharp in all cases, although it is nevertheless very close.}
\label{tab:asymptotic_p}
\end{center}
\end{table}

\subsection{Iteration counts and multigrid preconditioning}\label{sec:num_exp_2}
The condition number bound \eqref{eq:condition_number} holds for the preconditioner $\bo H$, which assumes that exact solvers are used to apply the inverses of the matrices $M+\tau \sqrt{\lambda_j}/2 A$. However, in practice, it is desirable to use a cheap approximation, such as a small number of iterations from an iterative solver for $M+\tau \sqrt{\lambda_j}/2 A$, which leads to an approximation of the ideal preconditioner $\bo H^{-1}$. 

\begin{table}
\begin{center}
\begin{tabular}{| c c | c c c c| }
\hline  & & & & &  \vspace{-2.5ex} \\
Mesh size & DoF & Direct & $1$ V-cycle & 2 V-cycles & 3 V-cycles \\ \hline 
 & & & & &  \vspace{-2.5ex} \\
$h=2^{-6}$ & 11 907  & 7 & 8 & 7 & 7  \\
$h=2^{-7}$ & 48 387  & 7 & 8 & 7 & 7  \\
$h=2^{-8}$ & 195 075 & 7 & 8 & 7 & 7  \\
$h=2^{-9}$ & 783 363 & 7 & 8 & 7 & 7 \\
$h=2^{-10}$& 3 139 587 & 7 & 8 & 7 & 7 \\ \hline 
\end{tabular}	\vspace{1ex}
\caption{Number of PCG iterations required for convergence of the PCG algorithm in the experiment of section~{\upshape\ref{sec:num_exp_2}}, for various mesh sizes and corresponding number of degrees of freedom. Fast convergence of PCG is observed in all cases, with the iteration counts being mesh-independent. The method using $1$ V-cycle appears as the most efficient in this experiment.}
\label{tab:multigrid_preconditioning}
\end{center}	
\end{table}

In this experiment, we study the effect of this approximation, in particular when the inverse of $M+\tau \sqrt{\lambda_j}/2 A$ is approximated by a small number of multigrid V-cycles. Let $M$ and $A$ be defined as the mass and stiffness matrices of the piecewise-linear finite element space defined on a uniform triangulation of size $h=2^{-k}$, $k=6,\dots,10$ in the two-dimensional domain $\Om = (0,1)^2$. The finest mesh thus leads to more than one million degrees of freedom (DoF) for $\bv$.
We consider the system $\bo L \bo u =\bo g$, with a chosen exact solution $\bo u$. We approximate solvers for $M+\tau \sqrt{\lambda_j}/2 A$ by applying $n$ V-cycles, $n\in\{1,2,3\}$, with symmetric Gauss--Seidel smoothers. Thus each application of the preconditioner requires $2n$ V-cycles per block.
To obtain a fair comparison of all the preconditioners, a relative tolerance of $10^{-6}$ of the true error in the energy norm was used to determine convergence of~PCG, and zero initial guesses were used for all computations.

Table~\ref{tab:multigrid_preconditioning} shows the PCG iteration counts for both ideal and approximate preconditioners; in this experiment, we set $p=2$ and $\tau = 0.1$, and thus there are more than 3 million degrees of freedom for $\Vp$ on the finest mesh.
Fast and robust convergence of PCG is observed in all cases, showing the effectiveness of these preconditioners. For this experiment, the method using $1$ V-cycle appears as the most efficient, as the reduction in cost per iteration outweighs the additional PCG iteration required.
\begin{table}
\begin{center}
\begin{tabular}{ | c | c  c |  c c  | }
\hline  & & & & \vspace{-2.5ex} \\
 & \multicolumn{2}{c}{$h=2^{-9}$} & \multicolumn{2}{c}{$h=2^{-10}$} \vspace{-2pt}\\ 
Order & Iterations & DoF & Iterations & DoF  \\ \hline 
$p=4$ & 8 & 1 305 605 & 8 & 5 232 645 \\
$p=6$ & 9 &1 827 847 & 9 &  7 325 703  \\
$p=8$ & 9 & 2 350 089 & 9 & 9 418 761 \\
$p=10$ & 9 & 2 872 331 & 9 & 11 511 819 \\
$p=12$ & 9 & 3 394 573 & 10 & 13 604 877 \\
$p=14$ & 9 & 3 916 815 & 9 & 15 697 935 \\  \hline             
\end{tabular}\vspace{1ex}
\caption{Dependence of the PCG iteration counts on the polynomial degrees $p=4,\dots,14$ in the experiment of section~{\upshape\ref{sec:num_exp_2}}, with approximate solvers using 1 multigrid V-cycle. The number of PCG iterations remains uniformly bounded, showing full robustness with respect to the polynomial degree, even up to 15 million DoF. The additional iteration for the case $p=12$, $h=2^{-10}$ is explained by the fact that the error on the $9$-th iteration was very close to, but above, the convergence criterion.}
\label{tab:p_dependence_multigrid}
\end{center}
\end{table}
In Table~\ref{tab:p_dependence_multigrid}, we study the dependence of the iteration counts on the polynomial degree. It is found that the use multigrid preconditioners retains the robustness of the preconditioner with respect to the approximation order $p$. 
 This is indeed the expected result, since Theorem~\ref{thm:eigenvalue_distribution} shows that as $p$ is increased, for large $j$ the matrices $M+\tau \sqrt{\lambda_j}/2 A$ are spectrally closer to $M$ and thus the efficiency of the multigrid V-cycle increases.
Our main goal in this experiment is to check the robustness of our preconditioners with respect to the polynomial degree, rather than to actually propose employing high-order approximations in time coupled with low-order approximations in space. Nevertheless, such high-order temporal approximations are encountered in the context of $hp$-version methods \cite[]{Schotzau2000,vonPetersdorff2004,Schotzau2010,WerderGerdesSchotzauSchwab2001}. 
We have also computed the results of this experiment using smaller time-step sizes $\tau$, and we observed a decrease in the iteration counts, as predicted by \eqref{eq:tau_asymptotics}.
The conclusion drawn from these results is that in practice, standard approximate solvers can be used while retaining the key properties of the ideal preconditioner $\bo H$, namely the fast convergence of the PCG algorithm and the robustness with respect to the discretization parameters.

\subsection{Parabolic problems with inexact solvers}\label{sec:num_exp_3}

The action of $\bo L$ as given in \eqref{eq:compute_L} requires the action of $A^{-1}$. In many practical applications, it is desirable to approximate this step by using an inexact solver, such as a fixed number of multigrid V-cycles, leading to an approximation $\bo{\hat{L}}\approx \bo L$. This raises the question of whether this can be performed without affecting firstly the accuracy of the solution, and secondly the performance of the preconditioning strategy. 
 This experiment provides evidence that inexact solvers can indeed be used without compromising these important objectives.

Consider the piecewise-linear simplicial conforming finite element approximation of the heat equation in the unit square $\Om=(0,1)^2$, final time $T=0.1$, imposed with initial datum $u_0(x,y) = x(1-x)\sin(\pi y)$ and with homogeneous Dirichlet lateral boundary conditions. This initial datum is chosen as it leads to a solution with decreased temporal regularity, see \cite[]{Schotzau2000}.
To compare near-exact and inexact solvers, we consider two approaches:
\begin{enumerate}
\item[(D)] direct solvers are used in the application of $\bo P^{\top}$, $\bo H^{-1}$ and $\bo L$,
\item[(MG)] 5 multigrid V-cycles are used to approximate the application of $A^{-1}$ in $\bo L$ and $\bo P^\top$, and 1 multigrid V-cycle is used to approximate $(M+\tau \sqrt{\lambda_j}/2 A)^{-1}$ in $\bo H^{-1}$, as in section~\ref{sec:num_exp_2}.
\end{enumerate}
We point out that the method (MG) does not use any direct solvers throughout the entire computation, and that each PCG iteration costs 7 V-cycles per block.

Table~\ref{tab:approximate_L_1} shows the final time errors $\norm{u(T)-u_\tau (T)}_{L^2(\Om)}$, for varying $\tau$, where $u$ denotes the exact solution of the PDE, and where $u_\tau$ denotes the discrete solution. Here, we used $h=2^{-8}$ and $p=1$. For comparison, we also show the errors attained for $p=0$, i.e.\ the backward Euler (BE) method.
The inexact solvers for (MG) retain the accuracy of the method, as the difference in solutions between (D) and (MG) is several orders of magnitude smaller than the difference to the exact solution. For $p=1$, the expected third order super-convergence rate is observed.

Table~\ref{tab:approximate_L_2} shows the average number of PCG iterations per time-step required by both approaches in order to obtain a residual tolerance of $10^{-6}$, as well as the final time error between the approximate solution $u_\tau$ for (D) and $\hat{u}_\tau$ for (MG). This shows that the number of PCG remains robust with respect to the approximation of $\bo L$ entailed by (MG).
This experiment shows that the use of inexact solvers can retain both the accuracy of the DG time-stepping method as well as the efficiency of the proposed preconditioners. 

\begin{table}
\begin{center}
\begin{tabular}{| c | c c c c |}
\hline  & & & & \vspace{-2.5ex} \\ 
$\tau/T$ & Error (BE) & Error (D) & Error (MG) & (D)--(MG) \\ \hline 
 & & & & \vspace{-2.5ex} \\ 
$ 1 $ & $2.546\times 10^{-2}$ & $3.078\times 10^{-3}$ & $3.078\times 10^{-3}$ & $1.303\times 10^{-8}$ \\
$ 1/2$ & $1.475\times 10^{-2}$ & $3.934\times 10^{-4}$ & $3.934\times 10^{-4}$ & $2.129\times 10^{-8}$ \\
$ 1/4$ & $8.008\times 10^{-3}$ & $5.444\times 10^{-5}$ & $5.441\times 10^{-5}$ & $3.219\times 10^{-8}$ \\
$ 1/8$ & $4.178\times 10^{-3}$ & $8.718\times 10^{-6}$ & $8.641\times 10^{-6}$ & $8.126\times 10^{-8}$ \\ \hline 
\end{tabular}
\vspace{1ex}
\caption{Final time errors $\norm{u(T)-u_\tau(T)}_{L^2(\Om)}$ obtained by the DG time-stepping method for the problem of section~{\upshape\ref{sec:num_exp_3}}, for $p=0$ (the backward Euler method) and for $p=1$ with direct solvers (D) and inexact multigrid (MG) solvers. The last column gives $\norm{u_\tau(T)-\hat{u}_\tau(T)}_{L^2(\Om)}$, where $u_\tau$, respectively $\hat{u}_\tau$, denotes the solution obtained by (D), respectively (MG).}
\label{tab:approximate_L_1}
\end{center}
\end{table}

\begin{table}
\begin{center}
\begin{tabular}{| c | c c c c | }	
\hline  & & & & \vspace{-2.5ex} \\ 
Iterations & $\tau =T$ & $\tau = T/2$ & $\tau =T/4$ & $\tau=T/8$ \\ \hline
(D) & 4 & 3.5 & 3 & 3 \\
(MG) & 6 & 5 & 5 & 5 \\ \hline 
\end{tabular}
\vspace{1ex}
\caption{Average number of PCG iterations per time-step for the problem of section~{\upshape{\ref{sec:num_exp_3}}}, with $p=1$  and with either direct (D) or inexact multigrid (MG) solvers.}
\label{tab:approximate_L_2}
\end{center}
\end{table}

\section*{Conclusion}
We have developed efficient and robust preconditioners enabling the fast solution of the DG time-stepping method by the preconditioned conjugate gradient algorithm. The analysis and numerical experiments show that the ideal and approximate preconditioners are robust with respect to all discretization parameters and lead to low condition numbers for the preconditioned system. Thus the high-order solution of large problems by the DG time-stepping method is tractable under the proposed approach.

\appendix
\section{Analysis of eigenvalues}\label{sec:eigenvalues}
In this section, we present the proof of Theorem~\ref{thm:eigenvalue_distribution}.
Without loss of generality, it is sufficient to consider only $p\geq 2$ throughout this section.
We will make use of Weyl's Theorem from eigenvalue perturbation theory~\cite[]{Demmel1997}.
\begin{theorem}[Weyl]
For a positive integer $n$, let $T$ and $\hat{T}$ denote symmetric $n\times n$ matrices. Let $\alpha_1\geq \dots \geq \alpha_n$ denote the eigenvalues of $T$ and let $\hat{\alpha_1}\geq \dots \geq \hat{\alpha_n}$ denote the eigenvalues of $\hat{T}$. Then, for each $j=1,\dots,n$, we have
\begin{equation}
\abs{\alpha_j - \hat{\alpha}_j} \leq \norm{T-\hat{T}}_{2},
\end{equation}
where $\norm{\cdot}_{2}$ denotes the matrix $2$-norm.
\end{theorem}

We are now ready to prove Theorem~\ref{thm:eigenvalue_distribution}.

\emph{Proof of Theorem~{\upshape\ref{thm:eigenvalue_distribution}}.}
Recall that the matrix $T$ is defined by \eqref{eq:T_def}. We start by noting that the matrix $T$ is pentadiagonal, i.e. $T_{kj} =0$ if $\abs{k-j}>2$. Moreover, it is easy to show from the properties of Legendre polynomials that there exists a constant $C$ independent of $p$ such that
\begin{equation}\label{eq:mag_psi_k}
\begin{aligned}
	\norm{\psi_k}_{L^2}^2 = T_{kk} \leq C (k+1)^{-2} & & &\forall\, k=0,\dots,p.
\end{aligned}
\end{equation}
The Cauchy--Schwarz inequality implies that $\abs{T_{jk}}\leq \sqrt{T_{kk} T_{jj}}$, and thus all entries of $T$ are uniformly bounded.
Since $T$ has at most $5$ nonzero entries per row, and all entries are uniformly bounded, the Gershgorin discs of $T$ are bounded independently of $p$. 
The Gershgorin Disc Theorem therefore implies that there is a constant $C$, independent of $p$, such that $\lambda_0$, the maximal eigenvalue of $T$, satisfies $\lambda_0 \leq C$. This corresponds to \eqref{eq:eigenvalue_bound} for the case $j=0$.
We consider now $j\geq 1$, and without loss of generality, we may assume that $p\geq 2$. For $0\leq q \leq p-1$, we define the orthogonal projector $\pi_{q}\colon \calP_p\tends \calP_{q}$ by
\begin{equation}\label{eq:projector_def}
\pi_q \colon \sum_{j=0}^p \underline{v}_j \,\psi_j \mapsto  \sum_{j=0}^{q} \underline{v}_j \,\psi_j,
\end{equation}
where the polynomials $\psi_j$ are defined by \eqref{eq:psi_definition}, and $\underline{v}=(\underline{v}_0,\dots,\underline{v}_p)\in \R^{p+1}$.
Define the matrix $\hat{T}_q \in \R^{(p+1)\times (p+1)}$ by $(\hat{T}_q)_{kj} \coloneqq (\pi_q \psi_k,\pi_q \psi_j)_{L^2}$. Note that $(\hat{T}_q)_{kj}$ is zero if either $k$ or $j$ is greater than $q$, and equals $T_{kj}$ otherwise. Thus, the matrix $\hat{T}_q$ has the general form
\begin{equation}\label{eq:hat_T_form}
\hat{T}_q = \begin{bmatrix}
	\tilde{T} & 0 \\ 0 & 0 
\end{bmatrix},
\end{equation}
where $\tilde{T}$ denotes the $(q+1)\times(q+1)$ principal submatrix of $T$.
The main step of the proof is to use \eqref{eq:mag_psi_k} repeatedly to find an upper bound for
\[
\norm{T-\hat{T}_q}_2 = \sup_{\underline{u}\in \R^{p+1}\setminus\{0\}}\sup_{\underline{v}\in \R^{p+1}\setminus\{0\}}\frac{\underline{v}^\top (T-\hat{T}_q ) \underline{u}}{\norm{\underline{u}}_2 \norm{\underline{v}}_2}.
\]
For arbitrary $\underline{u}$ and $\underline{v}\in \R^{p+1}\setminus\{0\}$, define the polynomials $u \coloneqq \sum_{j=0}^p \underline{u}_j \psi_j $ and $v = \sum_{j=0}^p \underline{v}_j \psi_j $. Then, we have
\begin{equation}
\begin{split}
\underline{v}^\top (T-\hat{T}_q ) \underline{u} & = (u,v)_{L^2} - (\pi_q u,\pi_q v)_{L^2} 
\\ & = (u-\pi_q u,v-\pi_q v)_{L^2} + (\pi_q u,v-\pi_q v)_{L^2} + (u-\pi_q u,\pi_q v)_{L^2}.
\end{split}
\end{equation}
It follows from the definition of $\pi_q$ in \eqref{eq:projector_def} that
\[
\norm{u-\pi_q u}_{L^2}^2 = \sum_{k=q+1}^p\sum_{j=q+1}^p \underline{u}_k\,  (\psi_k,\psi_j)_{L^2}\, \underline{u}_j.
\]
Since $(\psi_k,\psi_j)_{L^2}=T_{kj} =0$ if $\abs{k-j}>2$, the strengthened Cauchy--Schwarz inequality yields
\begin{equation}\label{eq:eigenvalue_bound_1}
\sum_{k=q+1}^p\sum_{j=q+1}^p \underline{u}_k  (\psi_k,\psi_j)_{L^2} \underline{u}_j \leq 5 \left( \sum_{k=q+1}^p  \norm{\psi_k}_{L^2}^2 \abs{\underline{u}_j}^2 \right)
 \leq \frac{C}{(q+1)^2} \norm{\underline{u}}_2^2.
\end{equation}
Note that this implies the following a priori estimate for the orthogonal projector:
\begin{equation}
\norm{u-\pi_q u}_{L^2} \leq  C (q+1)^{-1}\norm{ (\calI u)^\prime }_{L^2} \quad \forall\,u\in \calP_p.
\end{equation}
Thus \eqref{eq:eigenvalue_bound_1} and the Cauchy--Schwarz inequality imply that
\begin{equation}\label{eq:eigenvalue_bound_2}
\abs{(u-\pi_q u,v-\pi_q v)_{L^2}} \leq C (q+1)^{-2} \norm{\underline{u}}_2 \norm{\underline{v}}_2.
\end{equation}
For $k\leq q$ and $j\geq q+1$, we have $(\psi_k,\psi_j)=0$ if $k<q-1$ and $j>q+2$, and thus
\[
(\pi_q u,v-\pi_q v)_{L^2}=\sum_{k=0}^q \sum_{j=q+1}^p \underline{u}_k (\psi_k,\psi_j)\underline{v}_j = \sum_{k=\max(q-1,0)}^q \sum_{j=q+1}^{\min(p,q+2)} \underline{u}_k (\psi_k,\psi_j)\underline{v}_j.
\]
The Cauchy--Schwarz inequality thus implies that 
\begin{equation}\label{eq:eigenvalue_bound_3}
\begin{split}
\abs{(\pi_q u,v-\pi_q v)_{L^2}} & \leq 2 \left( \sum_{k=\max(q-1,0)}^q\abs{\underline{u}_k}^2 \norm{\psi_k}_{L^2}^2 \right)^{\frac{1}{2}}
\left(\sum_{j=q+1}^{\min(p,q+2)} \abs{\underline{v}_j}^2 \norm{\psi_j}_{L^2}^2\right)^{\frac{1}{2}} \\
&\leq  C (q+1)^{-2} \norm{\underline{u}}_{2} \norm{\underline{v}}_2.
\end{split}
\end{equation}
An identical argument shows that 
\begin{equation}\label{eq:eigenvalue_bound_4}
\abs{(u-\pi_q u,\pi_q v)_{L^2}}\leq  C (q+1)^{-2} \norm{\underline{u}}_{2} \norm{\underline{v}}_2.
\end{equation}
Combining \eqref{eq:eigenvalue_bound_2}, \eqref{eq:eigenvalue_bound_3} and \eqref{eq:eigenvalue_bound_4} implies that there exists a constant $C$, independent of $p$ and $q$, such that 
\begin{equation}
\norm{T-\hat{T}_q}_2 \leq  C (q+1)^{-2}.
\end{equation}
Therefore, Weyl's Theorem implies that the eigenvalues $\lambda_j$ of $T$ and $\hat{\lambda}_j$ of $\hat{T}_q$ satisfy
\begin{equation}
\begin{aligned}
\abs{	\lambda_j - \hat{\lambda}_j } \leq  C (q+1)^{-2} & & &\forall\, j=0,\dots,p,
\end{aligned}
\end{equation}
where the constant $C$ is independent of $j$, $q$ and $p$.
However, it is clear from \eqref{eq:hat_T_form} that $\hat{\lambda}_j = 0$ for $j\geq q+1$, and thus there exists a constant $C$ independent of $p$ and $q$ such that
\begin{equation}
\begin{aligned}
\lambda_j=\abs{\lambda_{j}-\hat{\lambda}_j} \leq  C (q+1)^{-2} & & & \forall\,j\geq q+1.
\end{aligned}
\end{equation}
The left-hand side of this inequality is independent of $q$ while the right-hand side is valid of all $q\leq j-1$. Therefore, the choice $q=j-1$ yields \eqref{eq:eigenvalue_bound} for $j=1,\dots,p$.

The lower bound \eqref{eq:eigenvalue_lower_bound} follows from inverse inequalities. Indeed, $\lambda_p$ satisfies
\[
\lambda_p = \min_{v\in\calP_p\setminus\{0\}} \frac{\norm{v}_{L^2}^2}{\norm{(\calI v)^\prime}_{L^2}^2}.
\]
To obtain \eqref{eq:eigenvalue_lower_bound}, it is therefore enough to show the inverse inequality
\begin{equation}\label{eq:calI_inverse_inequality}
\begin{aligned}
\norm{(\calI v)^\prime }_{L^2} \leq C (p+1)^2 \norm{v}_{L^2}	& & & \forall\,v\in\calP_p.
\end{aligned}
\end{equation}
For any $v\in \calP_p$, standard inverse inequalities \cite[]{Schwab1998} imply that
\begin{equation}\label{eq:eigenvalue_lower_bound_1}
\begin{split}
\norm{(\calI v)^\prime }_{L^2} & \leq \norm{v^\prime}_{L^2} + \norm{v}_{L_\infty} \tfrac{1}{2}\norm{(L_p-L_{p+1})^\prime}_{L^2}
\\ & \lesssim \left((p+1)^2 + (p+1) \,\norm{(L_p-L_{p+1})^\prime}_{L^2}\right) \norm{v}_{L^2}.
\end{split}
\end{equation}
It follows from \eqref{eq:derivative_legendre} that
\begin{equation}\label{eq:norm_derivative_legendre}
	\norm{(L_p-L_{p+1})^\prime}_{L^2}^2 = \sum_{j=0}^p \frac{8 (j+\tfrac{1}{2})^2}{2j+1} = 2\left(p+1\right)^2.
\end{equation}
Substituting \eqref{eq:norm_derivative_legendre} into \eqref{eq:eigenvalue_lower_bound_1} yields \eqref{eq:calI_inverse_inequality}.\qed

\section*{Acknowledgements}
The author wishes to express his thanks to Paul~Houston, Lorenz~John, Christian~Kreuzer, Endre~S\"uli and Martin~Vohral\'ik for many helpful discussions. The author was supported by an EPSRC Doctoral Prize at the Mathematical Institute, University of Oxford, during the period in which part of this work was completed.

\end{document}